\documentclass[11pt]{article}
\usepackage{geometry}                
\usepackage[parfill]{parskip}    
\usepackage{graphicx}
\usepackage{setspace}

\usepackage{amssymb}
\usepackage{amsmath}

\usepackage{bm}
\allowdisplaybreaks[3]

\newtheorem{theorem}{Theorem}[section]
\newtheorem{lemma}[theorem]{Lemma}
\newtheorem{proposition}[theorem]{Proposition}
\newtheorem{corollary}[theorem]{Corollary}

\newenvironment{proof}[1][Proof]{\begin{trivlist}
\item[\hskip \labelsep {\bfseries #1}]}{\end{trivlist}}
\newenvironment{definition}[1][Definition]{\begin{trivlist}
\item[\hskip \labelsep {\bfseries #1}]}{\end{trivlist}}
\newenvironment{example}[1][Example]{\begin{trivlist}
\item[\hskip \labelsep {\bfseries #1}]}{\end{trivlist}}
\newenvironment{remark}[1][Remark]{\begin{trivlist}
\item[\hskip \labelsep {\bfseries #1}]}{\end{trivlist}}

\newcommand{\qed}{\nobreak \ifvmode \relax \else
      \ifdim\lastskip<1.5em \hskip-\lastskip
      \hskip1.5em plus0em minus0.5em \fi \nobreak
      \vrule height0.75em width0.5em depth0.25em\fi}
      
\DeclareMathOperator{\End}{End}
\DeclareMathOperator{\Hom}{Hom}
\DeclareMathOperator{\id}{id}
\DeclareMathOperator{\Sym}{Sym}

\begin{document}

\begin{titlepage}
\title{Universal measuring coalgebras and R - transformation algebras}
\author{M. Batchelor$^a$ \and J. Thomas$^b$ \footnote{Supported by a Graduate Research Fellowship from the National Science Foundation, USA.}}

\maketitle

\begin{center}
$^a$ Corresponding Author\\
Dept. of Pure Mathematics and Mathematical Statistics,\\
Centre for Mathematical Sciences, University of Cambridge,\\
Wilberforce Road, Cambridge, CB3 0WB, United Kingdom\\
Tel:  +44 (0) 1223 765896, Fax: +44 (0) 1223 337920 , Email: mb139@cam.ac.uk
\end{center}

\begin{center}
$^b$ Email: jwthomas@post.harvard.edu
\end{center}

\bigskip
\begin{abstract}
\medskip
Universal measuring coalgebras provide an enrichment of the category of algebras over the category of coalgebras.  By considering the special case of the tensor algebra on a vector space $V$, the category of linear spaces itself becomes enriched over coalgebras, and the universal measuring coalgebra is the dual coalgebra of the tensor algebra $T(V\otimes V^*)$.    Given a braiding $R$ on $V$  the universal measuring coalgebra $P_R(V)$  which preserves the grading is naturally dual to the Fadeev-Takhtadjhan-Reshitikin bialgebra $A(R)$ and therefore provides a representation of the quantized universal enveloping algebra as an algebra of transformations.  The action of $P_R(V)$ descends to actions on quotients of the tensor algebra, whenever the kernel of the quotient map is preserved by the action of a generating subcoalgebra of $P_R(V)$.  This allows representations of quantized enveloping algebras as transformation groups of suitably quantized spaces.  


\end{abstract}
\end{titlepage}

\doublespacing

The technology of measuring coalgebras provides an enrichment of the category of algebras over coalgebras.  For forty years (\cite{kostant, marjmanifolds}) this technology has been employed in the spirit of algebraic geometry, to extend structures arising in differential topology to categories where the differential structure has been replaced by algebraic structure.  Chief among these applications have been to recover geometric interpretations of supermanifolds and Lie superalgebras.  More recently there have been attempts to interpret quantum groups in this framework \cite{majid, marjrapallo, marjadvances}.

The technique centres on the universal measuring coalgebra $P(A,B)$, which is a coalgebra that compares the structure of algebras $A$ and $B$.  Initially we hoped to represent quantized enveloping algebras as sub-bialgebras of universal measuring coalgebras $P(A,A)$, thus acting as endomorphisms of suitable ``function" algebras $A$.  This is the quantum version of the alternative construction of the universal enveloping algebra from a representation of the Lie algebra as derivations of $A$ \cite{marjrapallo, marjadvances}

In developing the technology to do this, a far simpler truth emerged.  The universal measuring coalgebra provides a seemingly harmless enrichment of the category of vector spaces over coalgebras.  Yet when $V$ is a vector space with a braiding $R: V \otimes V \to V \otimes V$, we can ask that the universal measuring coalgebra respects this additional structure.  The resulting ``$R$-transformation algebra" $P_R(V)$ is a bialgebra with the following desirable properties:

\begin{enumerate}
	\item $P_R(V)$ is defined by a universal property
	\item If $V$ is finite-dimensional, then $P_R(V)$ is canonically isomorphic to the dual of the Faddeev-Reshetikhin-Takhtajan (FRT) bialgebra $A(R)$.
	\item If V is a vector space on which the generators of a quantized enveloping algebra act appropriately, then $P_R(V)$ contains the quantized enveloping algebra.
	\item If $A = TV/J$ is an algebra on which the generators of a quantized enveloping algebra act appropriately, there is a homomorphism of bialgebras $P_R(V)$ to $P(A,A)$.
\end{enumerate}

Described in this setting, the quantized enveloping algebras are exactly transformations of a vector space which preserve a braiding, just as the orthogonal or symplectic algebras are transformations which preserve a form.

The plan of the paper is to present the minimum material on measuring coalgebras required to describe the construction, state the results concerning the representation of enveloping algebras as measuring coalgebras, and state and prove the main theorem which is the isomorphism of $P_R(V)$ with $A(R)^\circ$, the finite dual of the FRT bialgebra.  The proofs of the results on representations are contained in sections \ref{classicalembedproof} and \ref{quantizedalgebras}. The necessary results about measuring coalgebras are included in the appendix \ref{universalmeasuringcoalgebra}.  The sections are as follows:

\begin{enumerate}
	\item Basic definitions of measuring coalgebras.
	\item Applications to classical and quantized enveloping algebras.
	\item Measuring coalgebras and tensor algebras and their quotients.
	\item Definition and properties of $R$-transformation algebras.
	\item The representations of $U\mathfrak{g}$ in $P(A,A) $.
	\item The representations of $U_q\mathfrak {g}$ in  $P(A,A)$.
	\item Appendix  Properties of measuring coalgebras.
	\end{enumerate}

We would like to thank Martin Hyland for many helpful discussions.  We would also like to recognize the contribution of Ben Fairbairn. His extensive calculations in the preliminary stages of this work provided us with the experience on which to base our understanding.

\section{ Measuring coalgebras}

\subsection{Definitions} Let $C$ be a coalgebra.  We will use Sweedler's notation \cite{sweedler}, so the the comultiplication 

\[
 \triangle : C \to C\otimes C
  \]

is denoted

\[ 
\triangle c = \sum_{(c)} c_{(1)}\otimes c_{(2).}
\]

Let $A$, $B$ be algebras (over a field $k$ which will be either $\mathbb{R}$ or $\mathbb{C}$ throughout the paper).  A map $\sigma : C \to \Hom_k(A,B)$ is called a \emph{measuring map} if 

\[ \sigma (c) (aa') =  \sum_{(c)} \sigma (c_{(1)})(a)\sigma (c_{(2)})(a'),  \ \ \sigma (c) (1_A) = \epsilon (c)1_B\]
and
\[
\sigma (c)(1_A) = \epsilon (c)1_B
\]

for $a$, $a'$ in $A$ and $\epsilon $ the counit.  This is equivalent to requiring that the map

\[ \sigma : A \to \Hom(C,B)\]

is an algebra homomorphism, where the algebra structure on $\Hom(C,B)$ is the convolution product determined by the comultiplication on $C$.

$C$ together with $\sigma$ is called a \emph{measuring coalgebra}. 

If $(C',\sigma ')$ is another measuring coalgebra, a coalgebra map $ \rho :C \to C' $ is a \emph{morphism} of measuring coalgebras if  $\sigma = \sigma' \circ \rho $.  The category of measuring coalgebras for a pair of algebras (A,B) has a final object called the \emph{universal measuring coalgebra}.

\begin{proposition} \label{meascondition}Given a pair of algebras (A,B), there exists a measuring coalgebra P(A,B)
\[
\pi:P(A,B)\to \Hom(A,B)
\]
such that if $\sigma :C \to \Hom(A,B)$ is any measuring coalgebra, there exists a unique map $\rho :P(A,B) \to C$ such that $\pi \circ \rho = \sigma$.
\end{proposition}

The proofs are deferred to Appendix A.  The functorial properties are summarized in the following theorem.

\begin{theorem} \label{univmeasproperties}
\begin{enumerate}
		\item If $A, B, C$ are algebras, there is a map of coalgebras 
		\[
		P(A,B)\otimes P(B,C) \to P(A,C).
		\]
		\item The universal measuring coalgebra $P(*,*)$ is functorial in both variables.
		\item In particular $P(A,A)$ is a bialgebra and the measuring map is an algebra homomorphism. 
		\item If $A$ is a bialgebra, then $\Hom(A,B)$ is an algebra with the convolution product, $P(A, B)$ is a bialgebra, and the measuring map is an algebra homomorphism.
		\item $P(A,k) = A^\circ \subset Hom(A,k)$, the dual coalgebra of $A$.
\end{enumerate}
\end{theorem}

\section{Applications of the universal measuring coalgebra}

\subsection{Motivation: the embeddings of $U\mathfrak{g}$ and $U_q\mathfrak{g}$ }

Given that our ambition is to suggest a definition of $R$-transformation algebra which encompasses Lie groups, Lie algebras and quantized enveloping algebras, in this section and the next, we present the major results describing how universal enveloping algebras and their quantized equivalents can be represented as sub-bialgebras of universal measuring coalgebras.

The classical model that motivated the project is as follows. Let $\mathfrak g$ be a Lie algebra, and suppose that $A$ is algebra on which $\mathfrak g$ acts faithfully as derivations, so that there is a linear map $\phi: \mathfrak g \to \End(A)$ which is a homomorphism of Lie algebras.  This can be restated in terms of measuring coalgebras. Let $C = \mathbf C 1 \oplus \mathfrak g$ be given the structure of a coalgebra with $1$ group-like and elements of $\mathfrak g$ primitive.  The statement that elements of $\mathfrak g$ act as derivations is equivalent to saying that the extension of $\phi$ to all of $C$ sending $1$ to the identity in $End(A)$ measures. 

Let $\pi: P(A, A) \to \End(A)$ be the universal measuring map.  The observation is that the subbialgebra of $P(A,A)$ generated by $C$ is exactly the universal enveloping algebra of $\mathfrak g$.

\begin{theorem}\label{embed}
\begin{enumerate}
	\item The universal enveloping algebra $U\mathfrak g$ includes in $P(A,A)$ as a measuring bialgebra.
	\item Let $e:A \to k$ be an algebra homomorphism, and define $r = e\circ \phi:\mathfrak g \to \Hom(A,k)$. Suppose additionally that $r$ is injective on $\mathfrak g$.  If $A$ is a bialgebra, then $U\mathfrak g$ includes in $P(A,k)$ as a bialgebra.
	\item With $A$, $e$ as above, the map
	\[
	P(1,e): P(A,A) \to P(A,k),
	\]
	generated by $e \circ \pi: P(A, A) \to \Hom(A, k)$, sends $U\mathfrak g$ considered as a subalgebra of $P(A,A)$ isomorphically onto its image in $P(A,k)$.
\end{enumerate}
\end{theorem}

The proof is deferred to section \ref{classicalembedproof}. 

There are many examples of suitable algebras $A$ - the coordinate ring of $\mathfrak{g}$, polynomials on a vector space on which $\mathfrak{g}$ acts faithfully, and the exterior algebra on such a space all provide faithful representations of $U \mathfrak{g}$ as measuring coalgebras.  If $G$ is an algebraic group with Lie algebra $\mathfrak{g}$ and $k[G]$ is the coordinate ring of $G$, then $k[G]$ is a bialgebra, and so convolution in $\Hom(k[G],k)$ determines a bialgebra structure on $P(k[G],k)$, providing a representation of $U \mathfrak{g}$ in $P(k[G],k)$.

A statement analogous to \ref{embed} can be made replacing $U\mathfrak{g}$ by $U_q\mathfrak{g}$.  The construction above depends only on the fact that $C$, the generating set for $U\mathfrak{g}$, is a measuring coalgebra.  The generating set for $U_q \mathfrak g$ is also a coalgebra and the same technique will provide representations of $U_q\mathfrak{g}$ as sub-bialgebras of suitable $P(A,A)$ or $P(A,k)$. Details are given in section \ref{quantizedalgebras}.  Suitable algebras on which this coalgebra measures are constructed as quotients of tensor algebras of $U_q\mathfrak{g}$-modules.  It turns out that the tensor algebra and the universal measuring coalgebra play complementary roles.  Describing this relationship is the subject of the next section.

\section{Measuring coalgebras and tensor algebras and their quotients.}

The success of the project of recovering the classical transformation algebras and their quantized versions encouraged us to consider replacing the vector space of linear transformations $\Hom(V,W)$ by the universal measuring coalgebra $P(TV,TW)$ where $TV ($resp\ $ TW)$ is the tensor algebra on $V ($resp\ $ W)$. This provides an enrichment of the category of vector spaces over coalgebras.

\begin{proposition}\label{tensorV}
\begin{enumerate}
	\item
	Let $V, W$ be  vector spaces, and let $H$ a coalgebra together with a map
	\[
	\sigma:H \to \Hom(V,TW)
	\]
	Then $\sigma$ extends uniquely to a measuring map
	\[ 
	\hat{\sigma}:H \to \Hom(TV, TW). 
	\]
	\item  If H is in fact a bialgebra then $\hat{\sigma}$ is also an algebra homomorphism.
\end{enumerate}	
\end{proposition}	
	
\begin{proof}
\begin{enumerate}
\item The map $\sigma$ can be thought of as a map
\[
\sigma: V \to \Hom(H,TW).
\]
But since $T(V)$ is an algebra and $H$ is a coalgebra, $\Hom(H,TW)$ has the structure of an algebra, where
\[
\alpha \beta (b) = \sum_{(b)} \alpha (b_{(1)})\beta(b_{(2)}).
\]
The universal property of tensor algebras then extends $\sigma$ to an algebra map
\[
\sigma: TV \to \Hom(H,TW)
\]
or equivalently
\[
\hat{\sigma} : H \to \Hom(TV,TW)
\]
measures as required. 

\item The second part follows from the third part of \ref{univmeasproperties}, which is proved in the Appendix.
\end{enumerate}
\end{proof}

The object of interest here is  subcoalgebra of $P(TV,TW)$ whose elements determine maps from $V$ to $W$.

\begin{definition} Define
\[
P(V,W) = \{ p \in P(TV,TW): \pi(p)(V) \subset W \}.
\]
\end{definition}

(Here $\pi$ is the measuring map $\pi: P(TV,TW) \to \Hom (TV,TW)$.) It is easy to check that $P(V,W)$ is a subcoalgebra of $P(TV,TW)$. Replacing $\Hom(V,W)$ by  $P(V,W)$ gives an enrichment of the category of linear spaces over the category of coalgebras.  In the case of particular interest where $V = W$, $P(V,V)$ will be denoted simply by $P(V)$.  By the third part of \ref{tensorV}, $P(V)$ is a bialgebra.

Composition gives $\Hom(TV,TV) $ and hence $P(TV,TV)$ (resp $P(V)$) has the structure of a bialgebra.  

The other common multiplicative structure on $\Hom(D, A)$ is convolution when $D$ is a coalgebra and A is an algebra,  For any vector space $W$, the concept of a dual space can be "enriched" by considering the universal measuring coalgebra $P(TW,k)$.  When the vector space $W$ is replaced by a coalgebra $D$ this coalgebra becomes a bialgebra.

\begin{proposition}\label{tensorW}
\begin{enumerate}
	\item Let $C$ be any coalgebra, $W$ any vector space and let
	\[
	\mu: C \to \Hom(W,k)
	\]
	be any linear map.  Then $\mu$ extends uniquely to a measuring map
	\[ 
	\hat{\mu}: C \to \Hom(TW, k)
	\]
	\item If $D$ is a coalgebra, comultiplication in $D$ extends uniquely to a bialgebra structure on $TD$.
	\item Suppose $D$ is a coalgebra so that $Hom(D,k)$ is an algebra under convolution.  If $C$ is a bialgebra and $\mu$ is an algebra homomorphism, then $\hat{\mu}$ is also a bialgebra map from $C$ to $(TD)^\circ$.
\end{enumerate}
\end{proposition}

\begin{proof}
\begin{enumerate}
	\item As before, $\Hom(C,k)$ has the structure of an algebra (since $C$ is a coalgebra), so that 
	\[
	\mu: W \to \Hom(C,k)
	\]
	extends to 
	\[ 
	\hat{\mu}: TW \to \Hom(C,k)
	\]
	\item The inclusion $D\otimes D \subset TD \otimes TD$, provides a linear map
	\[
	\triangle : D \to TD \otimes TD.
	\]
	By the universal property of the tensor algebra, this extends uniquely to an algebra homomorphism 
	
	\[
	T\triangle : TD \to TD \otimes TD.
	\]
	Routine verification shows that $T\triangle$ is coassociative.  Similarly, the counit $\epsilon :D\to k$ extends to an algebra homomorphism $T: TD \to k$.  Again, the required identities can be verified by direct calculation.
	\item This is a direct application of the fourth part of \ref{univmeasproperties} which will be proved in the appendix.
	\end{enumerate}
\end{proof}

In the event that $W = V\otimes V^*$, observe that we can identify $V\otimes V^*$ in $T(V\otimes V^*)$ with $V\otimes V^*$ in $TV\otimes TV^*$, and hence make the inclusion
\[
T(V\otimes V^*) \subset TV\otimes TV^*.
\]

The case of interest is when $D = V\otimes V^*$, where there are isomorphisms
\[
\Hom(V,V) \\\ = \\\ V\otimes V^* \\\ = \\\ (V\otimes V^*)^*
\]
The Killing form identifies the space $V\otimes V^*$ with its dual.  Thus the vector space $V\otimes V^*$ carries both an algebra structure and a coalgebra structure, but these two structures are as incompatible as possible.  The function of $P(V)$ is to repair the incompatibility.  Moreover, the above identifications are as algebras (with the algebra structure on the last being given by convolution).  Thus the two apparently different cases of $\sigma$ and $\mu$ above in fact describe the same situation. 

\textbf{Notation}. It quickly becomes confusing whether $V\otimes V^*$ (or $TW$) is being regarded as an algebra or a coalgebra.  When there is danger of confusion, we will write $V\otimes_aV^*$ (or $T_aW$) when considering these spaces as an algebra, and $V\otimes_cV^*$ (or $T_cW$) when they are considered as coalgebras.
\begin{theorem}
\label{unification}
\[
P(V) = (T_a(V\otimes_c V^*))^\circ
\]
\end{theorem}

\begin{proof}
This is nearly tautological: by the definition of $P(V)$, the measuring map $\pi :P(V) \to \Hom(TV,TV)$ restricts to a linear map $\pi:P(V) \to \Hom(V,V)$.   But $\Hom(V,V) = V\otimes_a V^* = \Hom(V\otimes_c V^*, k)$.  Thus by the universal property of $P(T(V\otimes_c V^*), k)$,  the following diagram commutes

\begin{center}
\setlength{\unitlength}{1mm}
\begin{picture}(100,40)

	\put(30,10){\makebox(0,0){$P(T(V\otimes_aV^*,k)$}}
	\put(30,30){\makebox(0,0){$P(V)$}}
	\put(80,30){\makebox(0,0){$\Hom(V,V)$}}
	\put(80,20){\makebox(0,0){$\Hom(V\otimes_cV^*,k,)$}}
	\put(80,10){\makebox(0,0){$\Hom(T_a(V\otimes_c V^*),k)$}}
	
	\put(30,25){\vector(0,-1){10}}
	\put(37,30){\vector(1,0){31}}
	\put(80,25){\vector(0,-1){3}}
	\put(80,17){\vector(0,-1){3}}
	\put(50,10){\vector(1,0){10}}

	\put(55,33){\makebox(0,0){$\pi$}}
	\put(25,20){\makebox(0,0){$\rho$}}
	\put(55,13){\makebox(0,0){$\pi$}}
	\put(82,25){\makebox(0,0){$\cong$}}
\end{picture}	
\end{center}

and the map 
\[
\rho: P(V) \to   P(T(V\otimes V^*), k) = (T(V\otimes V^*))^\circ.
\]
is unique.
Conversely, the measuring map $\pi: P(T(V\otimes_c V^*), k) \to \Hom(T(V\otimes_c V^*), k)$ restricts to a linear map $\pi: P(T(V\otimes_c V^*), k) \to \Hom(V\otimes_c V^*, k) = \Hom(V,V).$ The resulting map of measuring coalgebras $P(T(V\otimes V^*), k) \to P(V)$ provides an inverse to $\rho$.

Since the identification

\[
\Hom(V\otimes _cV^*, k) \cong (V\otimes_a V* \cong \Hom(V,V)
\]
preserves the multiplicative structure of these spaces, the maps $\rho$ and its inverse are bialgebra maps.

\end{proof}

The maps induced by such $\sigma, \mu$ of \ref{tensorV}, \ref{tensorW} descend to quotients of $TV$ provided the ideal in question is preserved.  This is the chief tool for constructing enveloping algebras, both classical and quantized, as sub-bialgebras of universal measuring coalebras.

\begin{proposition}
\begin{enumerate}
\label{quotient}
	\item Suppose that $C$ is a coalgebra and that $\sigma:C\to \End(TV)$ measures.  Let $J$ be an ideal in $TV$.  Suppose additionally that  $\sigma(C)(J) \subset J$. Then $\sigma$ induces a measuring map $\tilde{\sigma}: C \to \End(TV/J)$ 
	
	\item If $\hat{\mu}: C \to \Hom(TV,k)$ measures, and $J$ is an ideal as above such that $c(J) = 0$ for all $c \in C$, then $\hat{\mu}$ descends to a measuring map $\hat{\mu}:C \to \Hom(TV/J,k)$.
\end{enumerate}
\end{proposition}
\begin{proof}
For both parts of the proposition the argument is the same.  The statement that $\sigma(C)(J) \subset J$ (resp $\hat{\mu}(C)(J) = 0$) says that $\sigma$ induces a map $\sigma :C \to End(TV/J, TV/J)$ (resp $\hat{\mu}$ descends to a map $\hat{\mu}: C \to Hom(TV/J, k)$).  These maps retain the measuring property.
\end{proof}

\section{$R$-transformation algebras}

The results of the previous sections have encouraged us to consider universal measuring bialgebras $P(V)$ as candidates for the algebra of transformations of a vector space $V$. Where the vector space becomes equipped with a specified braiding $R: V \otimes V \to V \otimes V$, we can describe the \emph{$R$-transformation algebra} $P_R(V)$ as the sub-bialgebra of $P(V)$ consisting of elements that preserve the braiding.  The $R$-transformation algebra incorporates the universal enveloping algebra and its quantized version as special cases.  Moreover, described in this setting, $P_R(V)$ is easily recognized as the dual coalgebra of the FRT bialgebra $A(R)$, when $V$ is finite-dimensional.  

\subsection{The braiding on $TV$.}

The bialgebra $P_R(V)$ will be constructed as a sub-bialgebra of $P(TV,TV)$ which preserves a braiding on $TV$ which extends the one given on $V$.

Recall that a \emph{braided algebra} $A$ is an algebra together with an invertible linear operator $\Psi: A \otimes A \to A \otimes A$ such that
\begin{enumerate}
 	\item $\Psi$  satisfies the braid equation on $A \otimes A \otimes A$, that is:
	\[
	(\id \otimes \Psi)(\Psi \otimes \id)(\id \otimes \Psi) = (\Psi \otimes \id)(\id \otimes \Psi)(\Psi \otimes \id).
	\]
	\item The multiplication $m: A \otimes A \to A$ and unit $\nu: \underline{1} \to A$ satisfy the following consistency conditions:
\begin{align}
\label{consistency}
&\Psi(m \otimes \id) = (\id \otimes m)(\Psi \otimes \id)(\id \otimes \Psi): A \otimes A \otimes A \to A \otimes A,\notag\\
&\Psi(\id \otimes m) = (m \otimes \id)(\id \otimes \Psi)(\Psi \otimes \id): A \otimes A \otimes A \to A \otimes A,\\
\Psi&(m \otimes \nu) = (\nu \otimes m): A \to A \otimes A, \qquad \Psi(\nu \otimes m) = (m \otimes \nu): A \to A \otimes A.\notag
\end{align}
\end{enumerate}
For further details on braided algebras and related structures, refer to \cite{majidreview}.

Now let $(V,R)$ be a \emph{braided vector space}, that is, $V$ is a vector space and $R: V \otimes V \to V \otimes V$ is an invertible linear operator satisfying the braid equation. Then $TV$ can be given the structure of a braided algebra.

\begin{proposition}If $R$ is a braid operator for $V$, then
\begin{enumerate}
	\item $R$ extends to a map $\Psi^{m,n}$
	\[  
	\Psi^{m,n} : \otimes ^mV \otimes \otimes^nV \to \otimes ^nV \otimes \otimes^mV.
	\]
	where $\Psi^{1,1}  = R: V \otimes V \to V \otimes V$.
	\item Writing $\Psi$ for $\sum_{m,n} \Psi ^{m,n}$, then $\Psi$ gives $TV$ the structure of a braided algebra.
	\item If $J$ is an ideal of $TV$ such that $\Psi J = J$, then $\Psi$ gives  $TV/J$ the structure of a braided algebra.
	
\end{enumerate}

\end{proposition}

\begin{proof}.  The map $\Psi^{m,n}$ is obtained by using $R$ on adjacent factors one pair at a time, and observing that the braid identity ensures that this is well-defined.  See Majid \cite{majidpoincare} for details.
\end{proof}

\subsection{Definition of $P_R(V)$}
\label{definition}

\begin{definition} Let $V$ be a vector space with braiding $R$.  Let $C$ be a coalgebra and suppose the map $\sigma: C \to \End(TV)$ is a measuring map.  Say that $C$ \emph{preserves} $(V, R)$ if:
\begin{align*}
&\sigma(c)(V) \subset V \qquad \forall c \in C,\\
R(\sigma \otimes \sigma)(\Delta(c))&(v \otimes w) = (\sigma \otimes \sigma)(\Delta(c))R(v \otimes w) \qquad \forall c \in C; \, \forall v, w \in V.
\end{align*}
\end{definition}
The next lemma follows from a simple application of equations \ref{consistency} and the braiding condition:

\begin{lemma} \label{maptoR}
Take $C$ a coalgebra and suppose the map $\sigma: C \to \End(TV)$ is a measuring that preserves $(V, R)$.  Then $C$ preserves the braiding $\Psi$ on all of $TV$, that is,
\begin{equation*}
\Psi(\sigma \otimes \sigma)(\Delta(c))(v \otimes w) = (\sigma \otimes \sigma)(\Delta(c))\Psi(v \otimes w) \qquad \forall c \in C; \,\, \forall v, w \in TV.
\end{equation*}
\end{lemma}

If $C$ is a measuring coalgebra which preserves $R$, then the sub-bialgebra generated by $C$ in $P(V)$ will also preserve $R$. If sub-bialgebras $C$ and $B$ of $P(TV, TV)$ preserve $R$, then so does the bialgebra generated by $C$ and $B$.  So, there is a largest sub-bialgebra of $P(TV, TV)$ which preserves $R$.  This is our preferred candidate for the role of $R$-transformation algebra.

\begin{definition}
The \emph{$R$-transformation algebra},  denoted  $P_R(V)$, of the braided vector space $(V, R)$, is the unique maximal sub-bialgebra of $P(V)$ which preseves $R$.
\end{definition}

From the remarks above, $P_R(V)$ is also the maximal subcoalgebra of $P(V)$ which preserves $R$ - the requirement that it be a bialgebra imposes no restriction.

The coalgebras measuring $TV$ to $TV$ and preserving $R$ form a full subcategory of the coalgebras measuring $TV$ to $TV$. The definition of $P_R(V)$ shows that it is the final object in this category. We have incidentally shown that $P_R(V)$ is a final object in another category. 
\begin{proposition} Let $\mathit{H}_R$ be the category whose objects are bialgebras $H$ together with actions of $H$ on $V$ such that the action on $V\otimes V$ (induced by comultiplication in $H$) preserves $R$.   Then $P_R(V)$ is also a final object in the category $\mathit{H}_R$.  

\end{proposition}

This will be useful in proving the duality between $P_R(V)$ and $A(R)$ in the following subsection.

\subsection{R -Admissable coactions and R-Admissable actions.}

The aim is to relate the $R$-transformation algebra $P_R(V)$ to the dual bialgebra of the Faddeev-Reshetikhin-Takhtajan (FRT) bialgebra associated with $R$, $A(R)$. To do so, we will describe in parallel a category in which $A(R)$ is an initial object and a category in which $P_R(V)$ is a final object.

The following lemma states that actions and coactions of bialgebras on a vector space extend to the whole tensor algebra.

\begin{lemma}
\label{extend(co)actions}
\begin{enumerate}
	\item Let $H$ be a bialgebra with an action $a: H\otimes V \to V$.  Comultiplication in $H$ extends this action to an action $Ta: H \otimes TV \to TV$.
	\item Let $B$ be a bialgebra with a coaction $ c:V \to V\otimes B$.  Multiplication in $B$ extends this coaction to a coaction $Tc: TV \to TV \otimes B$ which is an algebra homomorphism. 
\end{enumerate}
\end{lemma}

\begin{proof}
 The first statement is a restatement of the first part of \ref{tensorV}, and is only restated for comparison with the second.  The second makes use of the defining property of the tensor product.
\end{proof}

Very elementary algebra establishes the equivalence of an action $a: H \otimes V \to V$ with an algebra homomorphism $a: H \to V \otimes V^*$.  The coalgebra structure of $V^* \otimes V$, while equally elementary feels less familiar: comultiplication in $V^* \otimes V$ arises from the unit $u$ in the algebra $End(V) = V\otimes V^*$:
 \[
\bigtriangleup = 1\otimes u^* \otimes 1: V^* \otimes V \to V^* \otimes V \otimes V^* \otimes V.
 \]
 The counit $\epsilon$ is simply the evaluation map. The argument relating actions with homorphisms into $V^* \otimes V$ dualizes to give the result that coactions $c:V \to V \otimes B$ correspond to coalgebra maps $c: V^* \otimes V \to B$.

Using this definition of actions and coactions, the content of \ref{extend(co)actions} could be expressed by saying that the action $a$ (resp coaction $c$) extends to maps 

\begin{equation}
\label{tensor(co)action}
T^ia: H \to \otimes^i (V\otimes_a V^*), \ \ \  T^ic:\otimes(V^*\otimes_c V) \to B
\end{equation}
for all $i$.  

 We define \emph{admissable} actions and coactions to be those which preserve R, as follows.

\begin{definition} If $B$ be a bialgebra for which $V$ is a comodule, say the coaction $c:V \to V\otimes B$ is \emph{(V,R) admissable} (or simply admissable) if

\begin{center}
\setlength{\unitlength}{1mm}
\begin{picture}(100,40)

	\put(30,10){\makebox(0,0){$V\otimes V$}}
	\put(30,30){\makebox(0,0){$V\otimes V\otimes B$}}
	\put(80,30){\makebox(0,0){$V\otimes V\otimes B$}}
	\put(80,10){\makebox(0,0){$V\otimes V$}}
	
	\put(30,25){\vector(0,-1){10}}
	\put(40,30){\vector(1,0){27}}
	\put(80,25){\vector(0,-1){10}}
	\put(37,10){\vector(1,0){35}}

	\put(55,33){\makebox(0,0){$c$}}
	\put(25,20){\makebox(0,0){$R$}}
	\put(55,13){\makebox(0,0){$c$}}
	\put(87,20){\makebox(0,0){$R\otimes 1_B$}}
\end{picture}	
\end{center}

If $H$ is a bialgebra for which $V$ is a module, say the $a: H \otimes V \to B$ is \emph{(V,R) admissable} (or simply admissable) if 
\[
a \circ 1_H \otimes R = R \circ a : H\otimes V \otimes V \to V \otimes V.
\]

\begin{center}
\setlength{\unitlength}{1mm}
\begin{picture}(100,40)

	\put(30,10){\makebox(0,0){$V\otimes V$}}
	\put(30,30){\makebox(0,0){$H\otimes V\otimes V\otimes B$}}
	\put(80,30){\makebox(0,0){$H\otimes V\otimes V\otimes B$}}
	\put(80,10){\makebox(0,0){$V\otimes V$}}
	
	\put(30,25){\vector(0,-1){10}}
	\put(45,30){\vector(1,0){18}}
	\put(80,25){\vector(0,-1){10}}
	\put(37,10){\vector(1,0){35}}

	\put(55,33){\makebox(0,0){$c$}}
	\put(25,20){\makebox(0,0){$R$}}
	\put(55,13){\makebox(0,0){$c$}}
	\put(87,20){\makebox(0,0){$R\otimes 1_B$}}
\end{picture}	
\end{center}

\end{definition}

Using \ref{tensor(co)action}, the definition of admissable translates to a statement that $c$ factors through a coequalizer (resp. the image of $a$ lies in an equalizer).
 
 Denote by $\tau$ the map which twists the middle two factors
 \[
 \tau = 1_{V^*} \otimes twist \otimes 1_V : V^* \otimes V \otimes V^* \otimes V \to V^* \otimes V \otimes V^* \otimes V.
 \]
 Define two maps from $V^* \otimes V \otimes V^* \otimes V$ to itself:
 \[ 
 \alpha = \tau \circ R^* \otimes 1_{V\otimes V} \circ \tau,  \\\ \beta = \tau \circ 1_{V^* \otimes V^*} \otimes R \circ \tau.
 \]

Define $C$ to be the coequalizer of $\alpha$ and $\beta$, and $E$ to be the equalizer, thus
\begin{equation}
\label{coequalizer}
C = V^* \otimes V \otimes V^* \otimes V/im (\alpha - \beta)
\end{equation}
and 
\begin{equation}
\label{equalizer}
E = \{ w \in V^* \otimes V \otimes V^* \otimes V : \alpha w = \beta w \}.
\end{equation}
The translation of the statement that $a$ or $c$ preserves R can then be stated as a lemma.
\begin{lemma}
\begin{enumerate}
	\item A coaction $c$ of a bialgebra B on V is admissable if and only if 
	\[
	T^2c: (V^* \otimes V)\otimes (V^* \otimes V)  \to B 
	\]
	factors through C.
	\item An action $a$ of a bialebra H on $V^*$ is admissable if and only if 
	\[
	im T^2a(H) \subset E \subset (V\otimes V^*)\otimes (V^* \otimes V).
	\]
\end{enumerate}
\end{lemma}

\begin{proof}
Like all proofs of this nature, there is no difficulty beyond that of unravelling the definitions and displaying the material in a manner in which the claim becomes obvious.  We will show the statement for an action to be admissable.

The ingredients here are the following.  First, the comultiplication $H$ is simply

\begin{center}
\setlength{\unitlength}{1mm}
\begin{picture}(110,50)

	\put(30,15){\makebox(0,0){$V\otimes V^*$}}
	\put(30,35){\makebox(0,0){$H$}}
	\put(80,35){\makebox(0,0){$H\otimes H$}}
	\put(80,15){\makebox(0,0){$(V\otimes V^*)\otimes (V\otimes V^*)$}}

	\put(30,30){\vector(0,-1){10}}
	\put(35,35){\vector(1,0){35}}
	\put(80,30){\vector(0,-1){10}}
	\put(40,15){\vector(1,0){20}}

	\put(25,25){\makebox(0,0){$r$}}
	\put(55,18){\makebox(0,0){$c$}}
	\put(72,25){\makebox(0,0){$r\otimes r$}}
\end{picture}	
\end{center}

where $c$ is $1_V \otimes u_{V^*\otimes V} \otimes 1_V^*$.  The action of $H$ on $V$ is then

\begin{center}
\setlength{\unitlength}{1mm}
\begin{picture}(110,50)

	\put(30,15){\makebox(0,0){$V\otimes V^*\otimes V$}}
	\put(30,35){\makebox(0,0){$H\otimes V$}}
	\put(60,15){\makebox(0,0){$V$}}

	\put(30,30){\vector(0,-1){10}}
	\put(40,15){\vector(1,0){15}}
	
	\put(38,30){\vector(4,-3){15}}

	\put(25,25){\makebox(0,0){$r$}}
	\put(45,18){\makebox(0,0){$a$}}

\end{picture}	
\end{center}

where $a$ is contraction on the second and third factors.  Writing $(V\otimes V^*)\otimes (V\otimes V^*)$ as $V\otimes V\otimes V^* \otimes V^*$, the admissability condition can set in the following diagram
\begin{center}
\setlength{\unitlength}{1mm}
\begin{picture}(140,50)

	\put(20,25){\makebox(0,0){$H\otimes H\otimes V\otimes V$}}
	\put(70,15){\makebox(0,0){$V\otimes V \otimes V^* \otimes V^* \otimes V\otimes V$}}
	\put(70,35){\makebox(0,0){$V\otimes V \otimes V^* \otimes V^* \otimes V\otimes V$}}
	\put(120,35){\makebox(0,0){$V\otimes V$}}
	\put(120,15){\makebox(0,0){$V\otimes V$}}

	\put(80,30){\vector(0,-1){10}}
	\put(60,30){\vector(0,-1){10}}
	\put(70,30){\vector(0,-1){10}}
	\put(100,35){\vector(1,0){10}}
	\put(120,30){\vector(0,-1){10}}
	\put(100,15){\vector(1,0){10}}
	
	\put(35,27){\vector(3,2){10}}
	\put(35,23){\vector(3,-2){10}}

	\put(55,25){\makebox(0,0){$a$}}
	\put(65,25){\makebox(0,0){$b$}}
	\put(75,25){\makebox(0,0){$c$}}
	\put(115,25){\makebox(0,0){$R$}}
\end{picture}	
\end{center}

Here $a = R \otimes 1\otimes 1, b = 1 \otimes R^* \otimes 1$ and $c = 1 \otimes 1\otimes R$. The horizontal maps are the action - in this case given by contraction on the last four factors.  Evidently the diagram commutes if the map $a$ is used.  Admissability is the statement that the diagram also commutes for the map $c$.  But by the definition of $R^*$, this commutes if and only if the square with $b$ commutes.  Thus the diagram commutes if, on the image of $H\otimes H\otimes, a = b$.  But this is exactly the statement that the image of $H\otimes H$ in the coequalizer. 

\end{proof}

By its construction, $P_R(V)$ is readily seen to be the final object in the category of bialgebras $H$ equipped with a admissable action on $V$.  The initial object in the category of bialgebras $B$ equipped with admisable coactions on $V$ is exactly the FRT algebra.  To see this, it is only necessary to give the standard definition of the FRT algebra $A(R)$ in a basis independent fashion.

By \ref{tensorW} a coalgebra structure on $D$ endows the tensor algebra $TD$ with the structure of a bialgebra. Notice that if $K$ is a coideal in $TD$, the ideal generated by $K$ in $TD$ is also a coideal.  A routine calculation shows that $im(\alpha - \beta)$ is a coideal in $V^*\otimes V$. 

\begin{definition}
Let $J$ denote the coideal generated by $im(\alpha - \beta)$. Define $A(R)$, the FRT algebra, to be the bialgebra 
\[
A(R) = T(V^*\otimes V)/J.
\]
\end{definition}
The maps 
\[
j:V^*\otimes V \to A(R) 
\]

is evidently a coalgebra map.  Moreover, in the light of the preceding lemma, we have the following proposition.

\begin{proposition}
\label{admissable}
(The universal property of the FRT algebra.) If $B$ is a bialgebra with an admissable coaction $c$ on V, then there is a unique map 
\[
\rho :A(R) \to B
\]
such that 
\begin{center}
\setlength{\unitlength}{1mm}
\begin{picture}(110,50)

	\put(30,35){\makebox(0,0){$V$}}
	\put(30,15){\makebox(0,0){$V\otimes A(R)$}}
	\put(60,15){\makebox(0,0){$V\otimes B$}}

	\put(30,30){\vector(0,-1){10}}
	\put(40,15){\vector(1,0){13}}
	\put(38,30){\vector(4,-3){15}}
	
	\put(25,25){\makebox(0,0){$r$}}
	\put(45,18){\makebox(0,0){$a$}}

\end{picture}	
\end{center}

\end{proposition}

Thus $A(R)$ as constructed as above has the property expected of the FRT construction, and by uniqueness, this construction is equivalent to the more common basis dependent definition of $A(R)$.

The main theorem is the observation that \ref{unification} and \ref{admissable} combine to allow a simple identification of the $R$-transformation algebra and the FRT algebra.

\begin{theorem}
\[
P_{R^*}(V^*) = A(R)^\circ .
\]
\end{theorem}

\begin{remark}
\begin{enumerate}
	\item It is easy to verify that if $R:V\otimes V \to V\otimes V$ is a braiding (of a finite dimensional vector space) then the dual map $R^*:V^*\otimes V^* \to V^*\otimes V^*$ is also a braiding.
	\item Coactions on $V$ to correspond to actions on $V^*$ hence the theorem is stated in terms of $P_{R^*}(V^*)$.  However, $P_R(V) = P_{R^*}(V^*)^{op}$. See lemma \ref{conclusion}.
	\end{enumerate}
\end{remark}

\begin{proof}

From \ref{unification} we have that $P(V^*) = (T(V^*\otimes V))^\circ$. The essential step is to apply the content of \ref{admissable}.  

If $B$ is a bialgebra and $c:V \to V\otimes B$ is an admissable coaction, then 

\[
T^2c: (V^*\otimes V)\otimes (V^*\otimes V) \to B
\]

factors through $(V^*\otimes V)\otimes(V^*\otimes V)/J$ where $J$ is $ker(\alpha - \beta)$ from \ref{coequalizer}.  Thus

\[
c^\circ:B^\circ \to (V^*\otimes V)^* 
\]
has its image in $J^\perp$, that is, the action of $B^\circ$ is admissable.  Thus, in particular, the action of $A(R)^\circ$ on $V^*$ is admissable, and $A(R)^\circ \subset (T(V^*\otimes V)^\circ)$ is contained in $P_{R^*}(V)$.  

If $a:H \to V^*\otimes V$ is an admissable action of a bialgebra $H$ on $V^*$, then 
\[
T^2a: H \to (V^*\otimes V)\otimes(V^*\otimes V)
\]
\end{proof}

has its image in $E$ of \ref{equalizer}.  Then, since multiplication in$H$ induces $H^\circ \to H^\circ \otimes H^\circ$, the map $T^2a$ dualizes to give a coaction
\[
\psi =(T^2a)^\circ:  (V^*\otimes V)\otimes(V^*\otimes V)^* \to H^\circ.
\]
Since $T^2a$ had its image in $E$, $\psi$ factors through $C$ of \ref{coequalizer}.  In particular, this holds for $H = P_{R^*}(V)$, and hence $\psi$ factors through a map from $A(R)$ to $(P_{R^*}(V))$.

\begin{corollary}

  $P_R(V)^{op} = A(R)^\circ$

\end{corollary}

The corollary is an immediate consequence of the following lemma.

\begin{lemma}
\label{conclusion}
$P_R(V) \cong P_{R^*}(V^*)^{op}$.
\end{lemma}
\begin{proof}
This is simply the observation that nothing goes wrong when taking duals: if $\rho :C \to \End(V)$ is a linear map, define $\rho ':C \to \End(V^*)$ via $\rho ' (c) = (\rho (c))^*$.  If $\rho (c) $ preserved $R$, then $(\rho (c))^*$ will preserve $R^*$.  Thus as measuring coalgebras, $P_R(V) \cong P_{R^*}(V^*)$. Since $(\alpha \beta)^* = \alpha ^* \beta ^*$ for $\alpha , \beta$ in $\End(V)$ the order of multiplication is the reverse of the usual in  $P_{R^*}(V^*)$.
\end{proof}

\section{Proof of \ref{embed} for the classical case}
\label{classicalembedproof}
Recall that $\phi: \mathfrak{g} \to \End(A)$ is a faithful representation of $\mathfrak{g}$ as derivations on $A$.  Notice that $C = \mathbb{C}1 \oplus \mathfrak g$ can be given the structure of a coalgebra by setting $1$ to be grouplike, and elements in $\mathfrak g$ to be primitive.  Let $\hat{\phi}: C \to \End(A)$ map $1$ to the identity map on $A$ and equal $\phi$ when restricted to $\mathfrak{g}$.  Then $\hat{\phi}$ is an injective measuring map.

\begin{proof}[Proof of Theorem \ref{embed}]
1. Let $\pi: P(A, A) \to \End(A)$ be the universal measuring map.  By the universal property of measuring coalgebras there is a unique coalgebra map $\rho: C \to P(A,A)$ that satisfies $\pi \circ \rho = \hat{\phi}$.  Then, $\rho$ is injective because $\hat{\phi}$ is injective.  Let $U$ denote the sub-bialgebra of $P(A,A)$ generated by the image of $C$.

Consider the ideal $J$ in $P(A,A)$ generated by the set of elements of the form
\[\rho X \rho Y - \rho Y \rho X - \rho{[X,Y]}.\]

\begin{lemma} The ideal $J$ is in the kernel of the measuring map $\pi$.  Moreover, $J$ is also a coideal.
\end{lemma}
\begin{proof} Notice that since $\mathfrak g$ is represented in $\End(A)$ and since multiplication in $P(A,A)$ is defined via composition in $\End(A)$, $\pi (Z) = 0$ for all $Z$ in $\{\rho X \rho Y - \rho Y \rho X - \rho[X,Y]\}$.  Similarly, since $\pi$ is an algebra homomorphism, all of the ideal $J$ generated by elements $Z$ must also lie in the kernel of $\pi$.  
That $J$ is a coideal can be verified by checking directly that the space spanned by \[\rho X \rho Y - \rho Y \rho X - \rho{[X,Y]}\] is a coideal.
\end{proof}

\begin{lemma}\label{coideal}
If $J$ is any coideal in $P(A,A)$ which lies in the kernel of $\pi$, then $J = 0$.
\end{lemma} 

\begin{proof} If $J$ is a coideal in $P(A,A)$ which is contained in the kernel of $\pi$, then observe that $P(A,A)/J$ has the universal property which characterizes the universal measuring coalgebra, hence $J = 0$ by the uniqueness of $P(A, A)$.

Since $\rho$ is thus a linear map from $\mathfrak g$ to an associative algebra U which satisfies the identity 

\[\rho X \rho Y - \rho Y \rho X - \rho{[X,Y]} = 0,\]

by the universal property of universal enveloping algebras there is a unique algebra homomorphism 

\[\mu : U\mathfrak g \to U\]

which agrees with $\rho$ when restricted to $\mathfrak g$. But the map $\mu$ is not just an algebra homomorphism; it is a bialgebra map because $\rho$ is a coalgebra map.  

We now appeal to the following basic, if initially surprising, fact about coalgebra maps: they enjoy a rigidity that algebra homomorphisms lack.

\begin{proposition}\label{rigid}
Let $B$ be a coalgebra.  If $C$ is another coalgebra and $\nu: B \to C$ is a map of coalgebras, then $\nu$ is injective if and only if it is injective on the first coradical filtration of B. 
\end{proposition}

\begin{proof}\cite{montgomery} page 65. \end{proof}

It is thus sufficient to show that $\mu$ is injective on the first coradical filtration of $U\mathfrak{g}$, which is just $C$ \cite{sweedler}.  But $\mu = \rho$ on $C$, and is thus injective.

\end{proof}

2.  The structure of this part is parallel to that of the first part.  The same coalgebra $C$ can be used, and the measuring map from $C$ to $k$ is simply $r = e \circ \phi$.  This map measures, providing a map
\[
\hat{\rho}: C \to P(A,k).
\]   

As before, let $\hat{U}$ be the subalgebra of $P(A,k)$ generatd by the image of $C$, and take $\hat{J}$ to be the ideal generated by 

\[\hat{\rho} X\hat{\rho} Y - \hat{\rho} Y \hat{\rho} X - \hat{\rho}{[X,Y]}.\]

The proof continues as above, relying on the following two lemmas as before.

\begin{lemma}
The ideal $\hat{J}$ is in the kernel of the measuring map $\pi$.  Moreover, $\hat{J}$ is also a coideal.
\end{lemma}

\begin{lemma}
 If $\hat{J}$ is any coideal in $P(A,k)$ which lies in the kernel of $\pi$, then $\hat{J} = 0$
\end{lemma}

3.  Take $\pi: P(A, A) \to \End(A)$ and $\tilde{\pi}: P(A, k) \to \Hom(A, k)$ to be the respective universal measurings.  And let $\mu: U\mathfrak{g} \to P(A, A)$ be the injection from part 1 and let $\tilde{\mu}: U\mathfrak{g} \to P(A, k)$ be the injection from part 2.  Then $P(1, e) \circ \mu$ and $\tilde{\mu}$ provide two bialgebra maps from $U\mathfrak{g}$ to $P(A, k)$.  We must show that they are the same map.  By the universal property of $U\mathfrak{g}$, they are the same if they are the same on $\mathfrak{g}$.  By the universal property of $P(A, k)$, they are the same on $\mathfrak{g}$ if they give the same action when composed with $\tilde{\pi}$.  As maps from $\mathfrak{g}$ to $\Hom(A, k)$, these two maps satisfy:
\begin{align*}
&\tilde{\pi} \circ (P(1, e) \circ \mu) = (\tilde{\pi} \circ P(1, e)) \circ \mu = e \circ \pi \circ \mu = e \circ \phi,\\
&\tilde{\pi} \circ \tilde{\mu} = r = e \circ \phi,
\end{align*}
and the two maps are thus the same.\qed
\end{proof}

Particular choices of $A$ recover the universal enveloping algebra as the sub-bialgebra of $P(A,A)$ generated by $\mathfrak g$.  The following results follow easily from Theorem \ref{embed}.
\begin{corollary}\label{evaluation}
\begin{enumerate}
	\item $U \mathfrak g$ includes in $P(C^\infty (G), C^\infty (G))$.
	\item If $V$ is a faithful representation of $\mathfrak g$, then $U \mathfrak g$ includes in $P(TV, TV)$ where $TV$ is the tensor algebra on $V$
	\item If $V$ is a faithful representation of $\mathfrak g$, then $U \mathfrak g$ includes in $P(\Sym V,\Sym V)$ where $\Sym V$ is the symmetric algebra on $V$.
	\item If $V$ is a faithful representation of $\mathfrak g$, then $U \mathfrak g$ includes in $P(\bigwedge V, \bigwedge V)$ where  $\bigwedge V$ is the exterior algebra on $V$.
	\item Let $e$ denote the homomorphism $e: C^\infty (G) \to \mathbb{R}$ which evaluates a function at the identity $e$ of $G$.  Then

	\[P(1,e): P(C^\infty (G), C^\infty (G)) \to P(C^\infty (G), \mathbb{R})\]

	identifies the copy of $U \mathfrak g $ in $P(C^\infty (G) C^\infty (G))$ with the pointed subbialgebra of $P(C^\infty (G) C^\infty (G))$ with the identity as the unique group-like element.
\end{enumerate}

\end{corollary}

\begin{remark}
\begin{enumerate}
	\item Even if $S$ is just a linear space of derivations of $A$, the construction still generates the universal enveloping algebra of the Lie algebra generated by $S$.  In particular, if $\mathfrak{g}$ is semisimple, $U\mathfrak{g}$ will be generated in $P(A, A)$ as long as $\phi: \mathfrak{g} \to \End(A)$ is injective on the generators of $\mathfrak{g}$.  This is the classical parallel of the quantized situation discussed below.
	\item The third example demonstrates the fact that the universal enveloping algebra can be infinite dimensional even when the algebra A is finite dimensional.
\end{enumerate}
\end{remark}

In this way, the universal enveloping algebra is realized not as some external abstract construction appended to a Lie algebra, but is a very natural bialgebra of transformations (in the enriched setting) of a linear space. Moreover, this interpretation of the role of the universal enveloping algebra applies without further adjustment to the case of quantized enveloping algebras.  It is only necessary to introduce appropriate analogues of the symmetric and exterior algebras, and the function ring of a Lie group $G$.

\section{Quantized Enveloping Algebras}
\label{quantizedalgebras}
\subsection{Definitions and Basic Facts}

Let $\mathfrak{g}$ be a finite-dimensional complex semisimple Lie algebra of rank $n$ with Cartan matrix $(a_{ij})$ and let $d_i$ be the coprime positive integers such that the matrix $(d_i a_{ij})$ is symmetric.  Let $q$ be a fixed complex number, not a root-of-unity, and set $q_i \mathrel{\mathop:}= q^{d_i}$.  The algebra $U_q\mathfrak{g}$ is the complex associative algebra with $4n$ generators $E_i, F_i, K_i, K_i^{-1}, 1 \leq i \leq n$ and relations:
\begin{gather*}
\begin{aligned}
K_i K_j &= K_j K_i,  & K_i K_i^{-1} &= K_i^{-1} K_i = 1, \\
K_i E_j K_i^{-1} &= q_i^{a_{ij}} E_j, & K_i F_j &K_i^{-1} = q_i^{-a_{ij}} F_j,\\
\end{aligned}\\
\begin{aligned}
E_i F_j - F_j E_i &= \delta_{ij} \frac{K_i - K_i^{-1}}{q_i - q_i^{-1}}\\
\sum_{r = 0}^{1 - a_{ij}} (-1)^r \Biggl[ \Biggl[ \begin{matrix} 1 - a_{ij} \\ r \end{matrix} \Biggr] \Biggr]_{q_i}&E_i^{1 - a_{ij} - r} E_j E_i^r = 0, i \not= j\\
\sum_{r = 0}^{1 - a_{ij}} (-1)^r \Biggl[ \Biggl[ \begin{matrix} 1 - a_{ij} \\ r \end{matrix} \Biggr] \Biggr]_{q_i}&F_i^{1 - a_{ij} - r} F_j F_i^r = 0, i \not= j
\end{aligned}
\end{gather*}

where:
\begin{align*}
\Biggl[ \Biggl[ \begin{matrix} 1 - a_{ij} \\ r \end{matrix} \Biggr] \Biggr]_{q} &= \frac{[n]_q!}{[r]_q! [n - r]_q!}, && [n]_q = \frac{q^n - q^{-n}}{q - q^{-1}}.
\end{align*}

There is a Hopf algebra structure on $U_q\mathfrak{g}$ given by:
\begin{gather*}
\begin{aligned}
\Delta(K_i) &= K_i \otimes K_i, &\Delta(K_i^{-1}) &= K_i^{-1} \otimes K_i^{-1},
\end{aligned}\\
\begin{aligned}
\Delta(E_i) &= E_i \otimes K_i + 1 \otimes E_i, &\Delta(F_i) &= F_i \otimes 1 + K_i^{-1} \otimes F_i,
\end{aligned}\\
\begin{aligned}
\epsilon(K_i) &= 1, &\epsilon(E_i) &= \epsilon(F_i) = 0,
\end{aligned}\\
\begin{aligned}
S(K_i) &= K_i^{-1}, &S(E_i) &= -E_i K_i^{-1}, &S(F_i) &= -K_i F_i.
\end{aligned}
\end{gather*}

A quantum version of the PBW theorem \cite{pressley} shows that the generating coalgebra $C = \mathbb{C}[E_i, F_i, K_i, K_i^{-1}]_{1 \leq i \leq n}$ is a subcoalgebra of $U_q\mathfrak{g}$.

Chin and Musson \cite{musson} \cite{mussoncorr} and M\"{u}ller \cite{muller} have shown that the coalgebra structure of $U_q\mathfrak{g}$ is particularly simple:
\begin{proposition}
$U_q\mathfrak{g}$ is pointed with coradical $U_0 = \mathbb{C}[K_i, K_i^{-1}]_{1 \leq i \leq n}$ and first term of the coradical filtration $\sum_i U_0 + U_0 E_i + U_0 F_i$.
\end{proposition}

These authors use this result together with Proposition $\ref{rigid}$ to show:
\begin{proposition} \label{bi-ideal}
Every bi-ideal of $U_q\mathfrak{g}$ contains $E_i$ and $F_i$ for some $i$.
\end{proposition} 
\begin{remark} These results are only explicitly stated in \cite{musson} for simple $\mathfrak{g}$, although it follows from the work of \cite{muller} that they extend to semisimple $\mathfrak{g}$.  Also note that the result that we cite is slightly weaker than the result given in \cite{musson} (as extended by \cite{muller}).  It is all that we will need.\end{remark}

This provides a stronger version of Proposition $\ref{rigid}$:
\label{rigidquantized}
\begin{corollary}
Let $B$ be a bialgebra.  If $\nu: U_q\mathfrak{g} \to B$ is a map of bialgebras, then $\nu$ is injective if and only if it is injective on the generating coalgebra $C = \mathbb{C}[E_i, F_i, K_i, K_i^{-1}]_{1 \leq i \leq n}$.
\end{corollary}

We will also need the following result from the representation theory of quantized enveloping algebras.  Refer to \cite{pressley} for terminology and proof.  
\begin{proposition}
\label{braidedmodule}
The finite-dimensional type 1 modules of $U_q\mathfrak{g}$ carry a braiding that commutes with the action of $U_q\mathfrak{g}$, where this action is defined on $V \otimes V$ using the comultiplication in $U_q\mathfrak{g}$.
\end{proposition} 

\subsection{Statement and Proof of Theorem \ref{embed} in the Quantum Case}
Take $C = \mathbb{C}[E_i, F_i, K_i, K_i^{-1}]_{1 \leq i \leq n}$, the coalgebra of generators for $U_q\mathfrak{g}$.  Let $\phi: C \to \End(A)$ be a faithful measuring that preserves the defining relations of $U_q\mathfrak{g}$ in $\End(A)$, that is to say that $A$ can be made into a $U_q\mathfrak{g}$ module.
\begin{theorem}
\label{quantumembed}
\begin{enumerate}
	\item The quantized enveloping algebra $U_q\mathfrak g$ includes in $P(A,A)$ as a measuring 			bialgebra.
	\item Let $e:A \to k$ be an algebra homomorphism, and define $r = e\circ \phi: C \to \Hom		(A,k)$. Suppose additionally that $r$ is injective on $C$.  If $A$ is a bialgebra, then $U_q\mathfrak g$ includes in $P(A,k)$ as a bialgebra.
	\item With $A$, $r$ as above, the map
	\[
	P(1,e):P(A,A) \to P(A,k),
	\]
	generated by $e \circ \pi: P(A, A) \to \Hom(A, k)$, sends $U_q\mathfrak g$ considered as a subalgebra of $P(A,A)$ isomorphically onto its image in $P(A,k)$.
\end{enumerate}
\end{theorem}

The proof is the same as in the classical case with two modifications.  First, instead of Proposition $\ref{rigid}$, use Corollary $\ref{rigidquantized}$.  Secondly, instead of the universal property of $U\mathfrak{g}$, use the fact that a map $\phi: C \to \End(A)$ that preserves the defining relations of $U_q\mathfrak{g}$ in $\End(A)$ gives rise to a unique algebra map $\mu: U_q\mathfrak{g} \to \End(A)$ that agrees with $\phi$ on $C$.

\subsection{Examples of Algebras on which $U_q\mathfrak{g}$ measures}

Let $V$ be a module for $U_q\mathfrak{g}$ which is faithful on the generators.  Then by Proposition $\ref{tensorV}$, $U_q\mathfrak{g}$ measures $TV$ to $TV$ and the measuring map is an algebra homomorphism from $U_q\mathfrak{g}$ to $\End(TV)$.  By Theorem \ref{embed}, $U_q\mathfrak{g}$ embeds as a sub-bialgebra
\[
U_q\mathfrak{g} \subset P(TV, TV).
\]
The vector space $V$ has a braiding $R$ by Proposition \ref{braidedmodule} and $U_q\mathfrak{g}$ preserves $(V, R)$.  Thus, in the light of the discussion following \ref{maptoR}, $U_q\mathfrak{g}$ embeds as a sub-bialgebra
\[
U_q\mathfrak{g} \subset P_R(V).
\]
\begin{remark}
The identification of $P_R(V)$ with $A(R)^\circ$ means that $U_q\mathfrak{g}$ embeds in the dual of the FRT bialgebra of \emph{any} representation on which the generators of $U_q\mathfrak{g}$ act faithfully.
\end{remark}

There are also interesting quotient algebras of $TV$, on which $U_q\mathfrak{g}$ measures.  We construct these as follows.  Let $f$ be a complex polynomial, and suppose that $V$ is a vector space with braiding $R$ as above.  Following \cite{frt} define the algebra $\chi_{f, R}$  to be the quotient of the tensor algebra $TV$ by the two-sided ideal generated by $f(R)(V \otimes V)$.
\begin{remark}
This construction is possible for any operator $R: V \otimes V \to V \otimes V$, but choosing $R$ to be a braiding makes it less likely that the quotient will be trivial (see \cite{majidinternational} for details). Note that it might be necessary to choose an appropriate normalization on $R$ to obtain a non-trivial quotient.
\end{remark}

Note that if $\chi_{f, R}$ is non-trivial, then $V \subset \chi_{f, R}$.  By Proposition \ref{quotient}, $U_q\mathfrak{g}$ measures $TV$ to $TV$ and the measuring map is an algebra homomorphism from $U_q\mathfrak{g}$ to $\End(\chi_{f, R})$.  Thus, as long as $\chi_{f, R}$ is non-trivial, $U_q\mathfrak{g}$ embeds as a bialgebra
\[
U_q\mathfrak{g} \subset P(\chi_{f, R}, \chi_{f, R}).
\]

We list a few examples of non-trivial $\chi_{f, R}$.

\begin{example}Vector representation of $U_q\mathfrak{sl_n}$.

The vector representation $V$ of $U_q\mathfrak{sl_n}$ is $n$-dimensional and is the quantum version of the defining representation of $\mathfrak{sl_n}$.  Let $e_{i,j}$ be the $n \times n$ matrix with a $1$ in the $(i,j)$ position and $0$ elsewhere.  We can choose a basis such that V is given by:
\begin{multline*}
K_i = q^{-1} e_{i,i} + q e_{i+1,i+1} + \sum_{k \not= i, i+1} e_{k,k},\\
\shoveleft{\begin{aligned}
E_i &= e_{i+1,i}, & F_i &= e_{i,i+1}.
\end{aligned}}
\end{multline*}

The associated braiding (appropriately normalized) is:
\[
R = q \sum_{i} (e_{ii} \otimes e_{ii}) + \sum_{i \not= j} (e_{ii} \otimes e_{jj}) + (q - q^{-1}) \sum_{i < j} (e_{ij} \otimes e_{ji})
\]

With polynomial $x - q$ and this braiding, one obtains $\Sym_q\mathfrak{sl_n}$, generated by $1, x_1, ... x_n$ with relations:
\[
x_i x_j = q x_j x_i, \qquad i < j
\]
and with polynomial $x + q^{-1}$, one obtains $\bigwedge_q{\mathfrak{sl_n}}$ with the same generators and new relations:
\[
x_i^2 = 0 \qquad \text{and} \qquad x_i x_j = -q^{-1} x_j x_i, \qquad i < j.
\]
\end{example}

\begin{example}Symmetric algebra of $U_q\mathfrak{sp_4}$.

While the symmetric and exterior algebras of $U_q\mathfrak{sl_n}$ are very simple q-deformations of their classical counterparts, other quantized enveloping algebras give rise to more complex symmetric algebras.  As an interesting low-dimensional example, the algebra $\Sym_q\mathfrak{sp_4}$ is generated by $1, x_1, x_2, x_3,$ and $x_4$, subject to the relations:
\begin{align*}
x_1 x_2 &= q x_2 x_1, & x_1 x_3 &= q x_3 x_1,\\
x_2 x_4 &= q x_4 x_2, & x_3 x_4 &= q x_4 x_3,\\
x_1 x_4 &= q^2 x_4 x_1, & x_2 x_3 &= q^2 x_3 x_2 + (q - q^{-1}) x_1 x_4
\end{align*}
\end{example}

In fact, as in the above examples, for all classical simple $\mathfrak{g}$, one can use the vector representation to generate $\Sym_q\mathfrak{g}$ and $\bigwedge_q\mathfrak{g}$, which share many of the properties of symmetric and exterior algebras \cite{klimyk}, and it is clear from the above discussion that
\begin{align*}
&U_q\mathfrak{g} \subset P(\Sym_q\mathfrak{g}, \Sym_q\mathfrak{g}), && U_q\mathfrak{g} \subset P(\bigwedge_q\mathfrak{g}, \bigwedge_q\mathfrak{g}).
\end{align*}
As $\bigwedge_q\mathfrak{g}$ is finite-dimensional, it becomes readily apparent that one can find finite-dimensional $A$, such that $U_q\mathfrak{g}$ embeds in $P(A, A)$ for any semi-simple $\mathfrak{g}$ composed of classical simple Lie algebras, by using direct sums of appropriately chosen quantum exterior algebras.

\begin{example}Adjoint representation of $U_q\mathfrak{sl_2}$.

Other examples with rich structure can be found by using representations other than the vector representation.  For example, using the three-dimensional (``adjoint") type $\bm{1}$ representation of $U_q\mathfrak{sl_2}$ and taking an appropriate normalization of the braid matrix and a polynomial $f = x + q^{-2}$, one obtains for $\chi_{f, R}$ the algebra generated by $1, x_1, x_2, x_3$ with relations:
\begin{align*}
x_1 x_1 &= x_3 x_3 = 0, & x_2 x_2 &= \frac{q^2 - 1}{q + q^{-1}} x_1 x_3,\\
x_1 x_2 &= -q^{2} x_2 x_1, & x_2 x_3 &= -q^2 x_2 x_3, & x_1 x_3 = - x_3 x_1,\\
\end{align*}
which is better known as the quantum Lie algebra of $\mathfrak{sl_2}$ \cite{delius}.
\end{example}

\subsection{The quantized function algebras}

The algebras $\chi_{f, R}$ provide analogues of symmetric and exterior algebras of the (quantized) linear space on which a quantized enveloping algebra acts.  To complete the parallel with corollary \ref{evaluation}, we need an appropriate replacement for $C^{\infty}(G)$.  The following result motivates the definition of the bialgebra which will serve in this capacity.

\begin{definition}
With $G$ a connected, simply connected Lie group with Lie algebra $\mathfrak{g}$, the \emph{quantized function algebra} on $G$, denoted by $\mathbf{C}_q[G]$, is the Hopf subalgebra of the Hopf dual $(U_q\mathfrak{g})^\circ$ spanned by the matrix coefficients of all finite-dimensional type $\bm{1} \, U_q\mathfrak{g}$-modules.
\end{definition}

The dual pairing of bialgebras $\langle U_q\mathfrak{g}, \mathbf{C}_q[G] \rangle$ is derived from the pairing $\langle U_q\mathfrak{g}, U_q\mathfrak{g}^\circ  \rangle$.

This pairing provides a measuring map $\phi: U_q\mathfrak{g} \to \End(\mathbf{C}_q[G])$ which is also an algebra homomorphism, given by:
\[ 
\phi(u)(c) = \sum_{(c)} c_{(1)} \langle u, c_{(2)} \rangle \qquad u \in U_q\mathfrak{g}, c \in \mathbf{C}_q[G].
\]

This measuring is faithful on the generators of $U_q\mathfrak{g}$ and thus provides an embedding 
\[
U_q\mathfrak{g} \subset P(\mathbf{C}_q[G], \mathbf{C}_q[G]).
\]
Evaluating an element of $\mathbf{C}_q[G]$ on the identity of $U_q\mathfrak{g}$ provides an algebra homomorphism $e: \mathbf{C}_q[G] \to \mathbb{C}$ and composing $e \circ \phi: U_q\mathfrak{g} \to \Hom(\mathbf{C}_q[G], \mathbb{C})$ yields the dual pairing $\langle U_q\mathfrak{g}, \mathbf{C}_q[G] \rangle$.  This is certainly faithful on the generators of $U_q\mathfrak{g}$, so there is an embedding
\[
U_q\mathfrak{g} \subset P(\mathbf{C}_q[G], \mathbb{C}).
\]
An isomorphism between these embeddings is provided by the map:
\[
P(1,e): P(\mathbf{C}_q[G], \mathbf{C}_q[G]) \to P(\mathbf{C}_q[G], \mathbb{C}),
\]
as in Theorem \ref{quantumembed}.

\appendix 

\section{The universal measuring coalgebra}
\label{universalmeasuringcoalgebra}

\subsection{Construction of the universal measuring coalgebra}
While the linear dual of a coalgebra is an algebra, the full linear dual of an algebra is not necessarily a coalgebra: the categories of coalgebras and algebras are not equivalent.  In particular, coalgebras have a finiteness property with no corresponding property for algebras.  This finiteness property is critical for the construction of the universal enveloping algebra.

\begin{theorem}\label{finite} Every element of a coalgebra is contained in a finite dimensional subcoalgebra.
\end{theorem}

See Sweedler \cite{sweedler} or Montgomery \cite{montgomery} for a proof of this and other elementary coalgebraic constructions.

In outline the construction of $P(A,B)$ is as follows.

Consider the set of all finite dimensional measuring coalgebras, $ {C_\mu ,\sigma_\mu } $.  It is not difficult to verify that the coproduct  $ \coprod_\mu C_\mu$ is also a measuring coalgebra.  Moreover, if $ \rho :C_\lambda \to C_{\lambda '} $ is a measuring map, then there are two maps

\[ f,g: C_\lambda \to \coprod_\mu C_\mu\]

with $f$ being the inclusion of $C_\lambda $ in the $ \coprod_\mu C_\mu $ and $g$ being $\rho$ followed by the inclusion of $C_{\lambda'} $ in $ \coprod_\mu C_\mu $.The \emph{universal measuring coalgebra} $P(A,B)$ is then the co-equalizer
\[P(A,B) \to \coprod_\lambda C_\lambda\Longrightarrow^g_f \coprod_\mu C_\mu \].

The image of $P(A,k)$ is also referred to as the \emph{dual coalgebra of A} or the \emph{finite dual} of $A$, and is denoted by $A^\circ$.  

Most of the claims in \ref{univmeasproperties} follow easily from the universal property.  The third part follows from an inspection of the following diagram.  Suppose that $\Hom(A,B)$ is an algebra, and that $H$ is a bialgebra with a measuring map $\sigma :H \to \Hom(A,B)$ which is also an algebra homomorphism.   

\begin{center}
\setlength{\unitlength}{1mm}
\begin{picture}(110,60)

	\put(8,55){\makebox(0,0){$H\otimes H$}}
	\put(8,05){\makebox(0,0){$H$}}

	\put(30,20){\makebox(0,0){$P(A,B)$}}
	\put(30,40){\makebox(0,0){$P(A,B)\otimes P(A,B)$}}
	\put(80,40){\makebox(0,0){$\Hom(A,B)\otimes \Hom(A,B)$}}
	\put(80,20){\makebox(0,0){$\Hom(A,B)$}}
	
	\put(8,50){\vector(0,-1){40}}
	\put(20,55){\vector(4,-1){40}}
	\put(20,7){\vector(4,1){40}}

	\put(30,35){\vector(0,-1){10}}
	\put(50,40){\vector(1,0){5}}
	\put(80,35){\vector(0,-1){10}}
	\put(40,20){\vector(1,0){30}}
	
	\put(15,10){\vector(4,3){10}}
	\put(15,52){\vector(4,-3){10}}
	
	\put(45,53){\makebox(0,0){$\sigma\otimes \sigma$}}
	\put(25,30){\makebox(0,0){$m$}}
	\put(42,10){\makebox(0,0){$\sigma$}}
	\put(6,30){\makebox(0,0){$\mu$}}
	\put(17,15){\makebox(0,0){$\rho$}}
	\put(25,49){\makebox(0,0){$\rho \otimes  \rho$}}
	
	\put(55,23){\makebox(0,0){$\pi$}}
	
	\put(77,30){\makebox(0,0){$\mu$}}
\end{picture}	

\end{center}

The map $\mu \circ \sigma \otimes \sigma$ measures, and hence the map $m$ defining multiplication on $P(A,B)$ exists and is unique by the universal property of $P(A,B)$.

\begin{corollary}
\begin{enumerate}
	\item $P(A,k)$ includes in $\Hom(A,k)$
	\item Let $A$ and $H$ be bialgebras, and suppose that  $A$ is an $H$-comodule algebra.  Then $P(H,k) = H^\circ$ is a measuring coalgebra for the pair $(A,A)$.
\end{enumerate}
\end{corollary}

\begin{proof}

For 1, observe that $P(A,k) \subset \bigcup (A/J)^*$ where $J$ is a cofinite ideal in $A$.  Moreover, if $C$ is a finite dimensional subcoalgebra of $P(A,k)$, then there is an algebra homomorphism from $A$ to the finite dimensional algebra $C^*$, with cofinite kernel J. Since $P(A,k)$ is the union of its finite dimensional subcoalgebras, $P(A,k) \subset = (A/J)^*$.

For 2, observe that $H^\circ$ is a coalgebra with a map
\[
p: H^\circ \to \Hom(A,A), \qquad  p(\alpha)(a) = \sum_{(a)} a_{(0)} \alpha (a_{(1)}).
\]
 
Then
\begin{align*}
p(\alpha)(aa')  =  & \sum_{(a)(a')} a_{(0)}a'_{(0)}\alpha ( a_{(1)}a'_{(1)})  \\
		       =  &\sum_{(a)(a')(\alpha)} a_{(0)}a'_{(0)}\alpha_{(1)} (a_{(1)})\alpha_{(2)}(a'_{(1)}).
\end{align*}

But

\[
\sum_{(\alpha)}p(\alpha_{(1)}(a))p(\alpha_{(2)}(a'))  = \sum_{(\alpha) (a) (a')} a_{(0)}\alpha_{(1)}(a_{(1)})a'_{(0)}\alpha_{(2)}(a'_{(1)}),
\]

so $p$ measures, giving $H^\circ$ the structure of a measuring coalgebra as required.

\end{proof}

Majid \cite{majid} works with an arrows reversed version of this construction.  Instead of coalgebras measuring on algebras, one considers algebras comeasuring on algebras.  The category of comeasuring algebras for a pair of algebras $(A, B)$ does not have an intial object for arbitrary algebras $A$ and $B$, but when this universal comeasuring algebra does exist, Majid denotes it by $M(A, B)$.  This object is naturally dual to the universal measuring coalgebra $P(A, B)$, in that there is a natural isomorphism:
\begin{equation*}
M(A, B)^\circ \cong P(A, B).
\end{equation*}
This isomorphism is proved by noting that, as every element in a coalgebra is contained in a finite-dimensional coalgebra (Theorem \ref{finite}), it suffices to prove that $M(A, B)^\circ$ has the property \ref{meascondition} for every finite-dimensional $C$ measuring $A$ to $B$.  Then $C^\star$ is a comeasuring from $A$ to $B$, and, because $C^{\star \circ} \cong C$, the result follows from the universal property of $M(A, B)$.

\clearpage
\bibliography{thomasbatchelor_final.bib}
\bibliographystyle{amsplain}

\end{document}